\documentclass[11pt,reqno]{amsart}
\usepackage[T1]{fontenc}
\usepackage{lmodern,microtype}
\usepackage{amsmath,amssymb,amsthm,mathrsfs}
\usepackage[top=1.25in,bottom=1.22in,left=1.25in,right=1.25in]{geometry}
\usepackage{hyperref}
\hypersetup{hidelinks,pdfauthor={Shanlin Guan and Yuhang Zhao},pdftitle={Rigidity of Three-Dimensional Self-Shrinkers}}
\allowdisplaybreaks[1]
\numberwithin{equation}{section}
\newtheorem{theorem}{Theorem}[section]
\newtheorem{lemma}[theorem]{Lemma}
\newtheorem{proposition}[theorem]{Proposition}
\newcommand{\R}{\mathbb R}
\newcommand{\tr}{\operatorname{tr}}
\newcommand{\calR}{\mathcal R}
\newcommand{\calC}{\mathcal C}
\newcommand{\calL}{\mathcal L}
\newcommand{\hatY}{\widehat Y}
\title[Rigidity of three-dimensional self-shrinkers]{Rigidity of Three-Dimensional Self-Shrinkers with Constant Norm of the Second Fundamental Form}
\author[S. Guan]{Shanlin Guan}
\address{School of Mathematical Sciences, Chongqing Normal University, Chongqing 401331, P. R. China}
\email{20269003@cqnu.edu.cn}
\author[Y. Zhao]{Yuhang Zhao}
\address{School of Mathematics and Statistics, Hainan University, Haikou 570228, P. R. China}
\email{yuhangzhao@hainanu.edu.cn}
\thanks{Y. H. Zhao was supported by a start-up grant from Hainan University (XJ2600004145).}
\begin{document}
\begin{abstract}
Let $X:M^3\to\R^4$ be a complete, connected self-shrinker without boundary, and suppose that the squared norm $S$ of its second fundamental form is constant. We prove that $S=0$ or $S=1$. Accordingly, the image is a hyperplane through the origin or a generalized cylinder $\mathbb S^k(\sqrt{k})\times\R^{3-k}$, $1\le k\le3$, and the immersion is a Riemannian covering onto its image.
\end{abstract}
\maketitle
\enlargethispage{2pt}

\section{Introduction}

Self-shrinkers are self-similar solutions of mean curvature flow and arise naturally in the analysis of its singularities; see Huisken \cite{Huisken1990,Huisken1993} and Colding--Minicozzi \cite{ColdingMinicozzi2012}. Let $X:M^n\to\R^{n+1}$ be an immersed hypersurface, with second fundamental form $h$ and mean curvature $H=\tr h$ relative to a local unit normal $\nu$. We consider self-shrinkers satisfying
\begin{equation}\label{eq:shrinker}
H=-\langle X,\nu\rangle.
\end{equation}
Identifying $h$ with its shape operator, write $S=|h|^2$, $f_m=\tr(h^m)$, and $\sigma_2(h)=(H^2-S)/2$. The hyperplanes through the origin and the generalized cylinders $\mathbb S^k(\sqrt{k})\times\R^{n-k}$ are standard examples; they have $S=0$ and $S=1$, respectively. For the geometry of complete self-shrinkers, including the relation between properness and volume growth, see Ding and Xin \cite{DingXin2013} and Cheng and Zhou \cite{ChengZhou2013}; a drift maximum principle was established by Cheng and Peng \cite{ChengPeng2015}. 

The constant-$S$ classification problem is analogous to Chern's problem for closed minimal hypersurfaces in the unit sphere, which asks whether, for each fixed dimension, the possible constant values of $S$ form a discrete set; see \cite{XuXu2017}. In the spherical minimal setting, the scalar curvature is $R=n(n-1)-S$, and Simons' identity is $\frac12\Delta S=|\nabla h|^2+S(n-S)$. For Euclidean self-shrinkers, the corresponding identities are $R=H^2-S$ and $\frac12\calL S=|\nabla h|^2+S(1-S)$, where $\calL=\Delta-\langle X^T,\nabla\cdot\rangle$ on functions. The mean curvature in the Gauss equation and the drift term in Simons' identity distinguish the two problems. For constant $S$, the drift Simons identity yields $|\nabla h|^2=S(S-1)$. In particular, $0<S<1$ is impossible, while $S=1$ corresponds to parallel second fundamental form and hence to the generalized cylinders. The basic rigidity theorems for self-shrinkers were established by Le and \v{S}e\v{s}um \cite{LeSesum2011} and Cao and Li \cite{CaoLi2013}. Further pinching and rigidity results were obtained by Ding and Xin \cite{DingXin2014}, Xu and Xu \cite{XuXu2017}, Ding \cite{Ding2018},  Lei, Xu and Xu \cite{LeiXuXu2020}, the second author \cite{Zhao2025}, and  Li and the second author \cite{Zhao2026}. The classification problem for constant $S$ is therefore concentrated on the range $S>1$.

In dimension two, Ding and Xin \cite{DingXin2014} classified complete self-shrinkers with constant $S$ under polynomial volume growth; Cheng and Ogata \cite{ChengOgata2016} obtained the same classification for complete immersions. In dimension three, Cheng, Li and Wei proved the classification when $S$ and $f_4$ are constant \cite{ChengZLiWei2022}, and subsequently when $S$ and $f_3$ are constant \cite{ChengZLiWei2024}. In arbitrary dimensions, Cheng and Wei \cite{ChengWei2015} obtained rigidity for constant $S\le10/7$ under polynomial volume growth. Cheng, Wei and Yano \cite{ChengWeiYano2023} obtained a higher-dimensional gap theorem under constant $S$ and $f_3$; related rigidity results for $\lambda$-hypersurfaces appear in Cheng and Li \cite{ChengLi2026lambda}. More recently, Cheng, Li and Wei \cite{ChengLiWei2026} proved the classification for constant $S\le7/5$, and also for constant $S$ under the scalar-curvature bound $R=H^2-S\ge-7/5$. They subsequently classified complete self-shrinkers with positive constant scalar curvature \cite{ChengLiWei2026classification}. Cheng and Li \cite{ChengLi2026gap} proved the classification for constant $S\le10/7$, together with the corresponding result under constant $S$ and $R\ge-10/7$.

Brendle \cite{Brendle2016} proved that every compact embedded genus-zero self-shrinker in $\R^3$ is a round sphere. Andrews, Chow, Guenther and Langford \cite[Sec.~13.15.2, Open Problem~13.76]{AndrewsChowGuentherLangford2020} ask whether every compact embedded self-shrinker in $\R^4$ with the topology of $\mathbb S^3$ must be the round sphere $\mathbb S^3(\sqrt3)$. Cheng, Lai and Wei \cite[Remark~1.2]{ChengLaiWei2024} discuss the corresponding uniqueness conjecture for embedded self-shrinkers diffeomorphic to a sphere in higher dimensions. For complete three-dimensional self-shrinkers with constant $S$, we obtain the following classification.

\begin{theorem}\label{thm:main}
Let $X:M^3\to\R^4$ be a smooth, complete, connected self-shrinker without boundary. If $S=|h|^2$ is constant, then $S\in\{0,1\}$. Up to an orthogonal transformation, the image of $X$ is a hyperplane through the origin or
\[
\mathbb S^k(\sqrt{k})\times\R^{3-k},\qquad 1\le k\le3.
\]
In each case $X$ is a Riemannian covering onto its image.
\end{theorem}

In particular, a compact self-shrinker satisfying the hypotheses of Theorem~\ref{thm:main} has image $\mathbb S^3(\sqrt3)$. To exclude $S>1$, we first derive boundedness of the immersed image from the constant-$S$ identities and completeness. The Hessian maximum principle applied to $|X|^2/2$ then produces a limit at which $h>0$ and $Hh\ge I$. At this limit, the ordinary and $h$-weighted fourth-order identities concern the same completely symmetric tensor $U=\operatorname{Sym}(\nabla^2h)$. Combining these identities with the ordinary and $h$-weighted traces of $U$ yields a quadratic constraint $Q_h(G)\le0$ for the Gram tensor $G$ of $\nabla h$. The condition $Hh\ge I$ implies $\sigma_2(h)\ge1$, and the fixed-trace minimum of $Q_h$ is nonnegative, with equality only for $S=1$.

\section{Preliminaries}

Passing to a component of the two-sided normal cover preserves completeness, constant $S$, and the image. We may therefore work with a global unit normal. Let $g$ be the induced metric, write $I$ for the identity endomorphism, and use $\nabla$ both for the Levi--Civita connection and the induced tensor connections; $\bar\nabla$ denotes the Euclidean connection. Our conventions are
\[
\bar\nabla_Y\nu=-X_*h(Y),\qquad
\bar\nabla_Y(X_*Z)=X_*(\nabla_YZ)+h(Y,Z)\nu.
\]
For a smooth function or a covariant tensor, $\Delta=\tr\nabla^2$ is the rough Laplacian, and
\[
\calL=\Delta-\nabla_{X^T},
\]
where $X^T$ is the tangential projection of $X$; for functions the last term is $\langle X^T,\nabla\cdot\rangle$. At a fixed point with $h_{ij}=\lambda_i\delta_{ij}$, write $h_{ijk}=h_{ij,k}$ and $h_{ijkl}=h_{ij,kl}$, with the last index denoting the last derivative. All covariant derivatives are taken before diagonalizing $h$. Codazzi makes $h_{ijk}$ symmetric in its three indices; its covariant derivative $h_{ijkl}$ is symmetric in the first three, but generally not in the last two. Our curvature convention is $R^p{}_{iab}=h_{pa}h_{ib}-h_{pb}h_{ia}$. Indices range from $1$ to $n$, all component sums are displayed explicitly, and tensor norms count full ordered index sums. Inequalities between symmetric two-tensors are understood as inequalities of quadratic forms.

The symmetric two-tensor $G$, with $G_{kl}=\langle\nabla_{e_k}h,\nabla_{e_l}h\rangle$, is the Gram tensor of $\nabla h$ and is positive semidefinite. The contractions used below are
\begin{equation}\label{eq:contractions}
\begin{gathered}
N=|\nabla h|^2,\qquad G_{kl}=\sum_{i,j}h_{ijk}h_{ijl},\qquad \tr G=N,\\
A=\sum_{i,j,k}\lambda_i^2h_{ijk}^2=\tr(h^2G),\quad
B=\sum_{i,j,k}\lambda_i\lambda_jh_{ijk}^2,\quad
C=\sum_{i,j,k}\lambda_i h_{ijk}^2=\tr(hG).
\end{gathered}
\end{equation}
The trace expressions for $A$ and $C$ use the complete symmetry of $h_{ijk}$.
Differentiating $H=-\langle X,\nu\rangle$ and using $\bar\nabla_Y\nu=-X_*h(Y)$ gives $\nabla H=h(X^T,\cdot)$. As $\nabla_YX^T=Y-Hh(Y)$, a further derivative and Codazzi yield
\begin{equation}\label{eq:gradH}
\nabla H=h(X^T,\cdot),\qquad
\nabla^2H=\nabla_{X^T}h+h-Hh^2.
\end{equation}
For $\rho=|X|^2/2$, it follows that
\begin{equation}\label{eq:rho}
\nabla\rho=X^T,\qquad \nabla^2\rho=g-Hh.
\end{equation}
We record the drift identities in these conventions; compare \cite{ChengWeiYano2023,DingXin2014,ChengLiWei2026}.

\begin{lemma}\label{lem:drift}
For a self-shrinker satisfying \eqref{eq:shrinker},
\begin{equation}\label{eq:drift}
\begin{gathered}
\calL h=(1-S)h,\qquad \calL H=(1-S)H,\\
\tfrac12\calL S=|\nabla h|^2+S(1-S),\\
\tfrac13\calL f_3=(1-S)f_3+2C,\qquad
\tfrac14\calL f_4=(1-S)f_4+2A+B.
\end{gathered}
\end{equation}
If $S$ is constant, then
\begin{equation}\label{eq:constantS}
N=S(S-1),\qquad \sum_i\lambda_i h_{iik}=0,
\end{equation}
and
\begin{equation}\label{eq:pengterng}
|\nabla^2h|^2=N(S-2)+3(A-2B).
\end{equation}
\end{lemma}
\begin{proof}
The Simons equation $\Delta h=\nabla^2H+Hh^2-Sh$, combined with \eqref{eq:gradH}, gives $\calL h=(1-S)h$. Its trace proves the identity for $H$. For $h_k=\nabla_{e_k}h$, the product rule applied to $f_m=\tr(h^m)$, $m\ge2$, yields
\[
\frac1m\calL f_m=(1-S)f_m+
\sum_k\sum_{a=0}^{m-2}\tr(h^ah_kh^{m-2-a}h_k).
\]
For $m=3$, the last sum is $2\sum_{i,j,k}\lambda_i h_{ijk}^2=2C$; for $m=4$, the three summands give $A+B+A$. The case $m=2$ gives the identity for $S$. If $S$ is constant, then
\[
0=\tfrac12S_k=\sum_i\lambda_i h_{iik},
\qquad 0=\tfrac12\calL S=N+S(1-S),
\]
which establishes \eqref{eq:constantS}.

For the second derivative identity, put $v=X^T$. At a principal frame, differentiating $\calL h=(1-S)h$ and commuting the Laplacian and covariant derivative gives
\begin{equation}\label{eq:drift-derivative-commute}
\begin{aligned}
(\calL\nabla h)_{ijk}={}&(2-S)h_{ijk}-H\lambda_kh_{ijk}-h_{ij}S_k
+\mathscr K_{ijk}\\
&-\sum_{a,p}v^p\big(R^a{}_{ikp}h_{aj}+R^a{}_{jkp}h_{ia}\big),
\end{aligned}
\end{equation}
where $\mathscr K_{ijk}=\sum_a(h_{ij,kaa}-h_{ij,aak})$ and $\nabla_kv^p=\delta_{kp}-Hh_{kp}$. The Ricci commutation formula for the second and third covariant derivatives is
\begin{equation}\label{eq:third-order-ricci}
\begin{aligned}
h_{ij,kaa}-h_{ij,aak}
=\sum_p\big[& (\nabla_aR^p{}_{ika})h_{pj}+(\nabla_aR^p{}_{jka})h_{ip}\\
&+2R^p{}_{ika}h_{pja}+2R^p{}_{jka}h_{ipa}
+R^p{}_{aka}h_{ijp}\big].
\end{aligned}
\end{equation}
With this curvature convention, the commutator in \eqref{eq:drift-derivative-commute} becomes, in a principal frame,
\[
\begin{aligned}
\mathscr K_{ijk}={}&\big[2\lambda_i\lambda_j-\lambda_i^2-\lambda_j^2
+2\lambda_k(\lambda_i+\lambda_j)+\lambda_k(H-\lambda_k)\big]h_{ijk}\\
&+\delta_{jk}\big[\lambda_j(\lambda_j-\lambda_i)H_i-\lambda_jS_i\big]\\
&+\delta_{ik}\big[\lambda_i(\lambda_i-\lambda_j)H_j-\lambda_iS_j\big],
\end{aligned}
\]
where $H_i=\lambda_i v_i$ and $S_i=2\sum_a\lambda_ah_{aai}$. This follows by inserting Gauss and its derivative in \eqref{eq:third-order-ricci}. The last term of \eqref{eq:drift-derivative-commute} is
\[
-\sum_{a,p}v^p\big(R^a{}_{ikp}h_{aj}+R^a{}_{jkp}h_{ia}\big)
=-\delta_{jk}\lambda_j(\lambda_j-\lambda_i)H_i
-\delta_{ik}\lambda_i(\lambda_i-\lambda_j)H_j.
\]
The terms containing $H_i,H_j$ cancel, and the remaining coefficients simplify to
\[
\begin{aligned}
(\calL\nabla h)_{ijk}={}&\Big(2-S+(\lambda_i+\lambda_j+\lambda_k)^2
-2(\lambda_i^2+\lambda_j^2+\lambda_k^2)\Big)h_{ijk}\\
&-h_{ij}S_k-h_{jk}S_i-h_{ki}S_j.
\end{aligned}
\]
For its contraction with $h_{ijk}$, complete symmetry gives
\[
\begin{aligned}
&\sum_{ijk}\big[(\lambda_i+\lambda_j+\lambda_k)^2
-2(\lambda_i^2+\lambda_j^2+\lambda_k^2)\big]h_{ijk}^2
=-3A+6B,\\
&\sum_{ijk}(h_{ij}S_k+h_{jk}S_i+h_{ki}S_j)h_{ijk}
=3\sum_k S_k\sum_i\lambda_i h_{iik}=\tfrac32|\nabla S|^2.
\end{aligned}
\]
The drift product rule for $|\nabla h|^2$ therefore yields
\[
\tfrac12\calL N
=|\nabla^2h|^2+(2-S)N-3(A-2B)-\tfrac32|\nabla S|^2.
\]
For constant $S$, one has $\calL N=0$ and $\nabla S=0$, so rearranging gives \eqref{eq:pengterng}.
\end{proof}

We recall the Hessian maximum principle in the form needed below (cf. \cite{Omori1967,Petersen2016}).
\begin{lemma}\label{lem:omori}
Let $(M,g)$ be complete, connected and without boundary, with sectional curvature bounded below. If $u\in C^2(M)$ is bounded below, there are points $p_j$ and $\varepsilon_j\downarrow0$ such that
\[
u(p_j)\to\inf_Mu,\qquad
|\nabla u|(p_j)\le\varepsilon_j,\qquad
\nabla^2u(p_j)\ge-\varepsilon_jg.
\]
For a bounded-above function, the inequalities are reversed for the Hessian and the values approach the supremum.
\end{lemma}
\begin{proof}
Fix $o\in M$ and let $r(p)=d(o,p)$. Choose $\kappa\ge0$ such that $\sec_M\ge-\kappa^2$. We take
\[
\psi(p)=\sqrt{1+r(p)^2}.
\]
Away from $o$ and its cut locus, Hessian comparison yields
\[
\nabla^2r\le\kappa\coth(\kappa r)(g-dr\otimes dr)
\le(1/r+\kappa)(g-dr\otimes dr),
\]
with $\kappa\coth(\kappa r)=1/r$ if $\kappa=0$. The derivatives of the function of $r$ defining $\psi$ are
\[
\psi'(r)=\frac{r}{\sqrt{1+r^2}},\qquad
\psi''(r)=\frac1{(1+r^2)^{3/2}}.
\]
Consequently
\begin{equation}\label{eq:omori-barrier}
|\nabla\psi|\le1,\qquad
\nabla^2\psi=\psi''dr\otimes dr+\psi'\nabla^2r
\le(1+\kappa)g.
\end{equation}
Here the radial coefficient is at most $1$ and the tangential coefficient is at most $(1+\kappa r)/\sqrt{1+r^2}\le1+\kappa$. The function $\psi$ is smooth near $o$, since $r^2$ is smooth there.

At a cut point $p\ne o$, let $\gamma:[0,r(p)]\to M$ be a minimizing geodesic from $o$ to $p$. For sufficiently small $\eta>0$, the Calabi distance $r_\eta(x)=\eta+d(\gamma(\eta),x)$ is smooth near $p$ and is an upper support for $r$ there. Thus $\sqrt{1+r_\eta^2}$ is an upper support for $\psi$. Applying Hessian comparison to this support, and then letting $\eta\downarrow0$, gives the same bound \eqref{eq:omori-barrier} for the minimum test at $p$. Indeed, if $u+\delta\psi$ has a minimum at $p$, then $u+\delta\sqrt{1+r_\eta^2}$ has a local minimum there as well, since the support lies above $\psi$ and agrees with it at $p$.

For $\delta>0$, the function $u+\delta\psi$ attains its minimum at a point $p_\delta$, since $u$ is bounded below and $\psi$ is proper by Hopf--Rinow. The first- and second-derivative tests, using the upper support at cut points, give
\begin{equation}\label{eq:omori-estimates}
|\nabla u|(p_\delta)\le\delta,\qquad
\nabla^2u(p_\delta)\ge-(1+\kappa)\delta g.
\end{equation}
For any fixed $q\in M$, minimality and $\psi\ge0$ yield
\[
\inf_Mu\le u(p_\delta)
\le u(q)+\delta\big(\psi(q)-\psi(p_\delta)\big)
\le u(q)+\delta\psi(q).
\]
Let $\delta\downarrow0$ and then take the infimum over $q$ to obtain $u(p_\delta)\to\inf_Mu$. Absorbing $1+\kappa$ into a sequence $\varepsilon_j\downarrow0$ proves the assertion. The bounded-above case follows by applying it to $-u$.
\end{proof}

For constant $S$, the Gauss equation gives $\sec_M\ge-S/2$. Thus the Hessian maximum principle applies to the position function $\rho=|X|^2/2$ once its boundedness has been established.

\section{A positive limiting spectrum}

Henceforth $n=3$. The identity \eqref{eq:pengterng} determines the ordinary norm of $\nabla^2h$. A contraction with $h$ gives a second fourth-order identity, whose leading term becomes a squared norm at the positive limit. Put
\[
A_3=\sum_{i,j,k}\lambda_i^3h_{ijk}^2,\qquad
B_3=\sum_{i,j,k}\lambda_i^2\lambda_jh_{ijk}^2,
\qquad \mathscr J=\sum_{i,j,k,l,m}h_{ijkl}h_{ijm}h_{klm}.
\]
For a covariant tensor $T$, define $D_hT=\sum_{a,b}h_{ab}(\nabla^2T)(\cdot,e_a,e_b)$, where the dots denote the tensor slots of $T$.
\begin{lemma}\label{lem:fourth}
If $S$ is constant, then
\begin{equation}\label{eq:fourth}
\sum_{i,j,k,l}\lambda_lh_{ijkl}^2-3\mathscr J
=f_3N-3SC+3A_3-6B_3.
\end{equation}
\end{lemma}
\begin{proof}
The Hessian of the constant function $S$ gives, using $\frac12S_{,ij}=\sum_{a,b}h_{abi}h_{abj}+\sum_a\lambda_a h_{aaij}$,
\[
\sum_a\lambda_a h_{aaij}=-G_{ij}.
\]
In a principal frame, Codazzi symmetry and Ricci commutation give
\[
\begin{aligned}
h_{ijaa}-h_{aaij}&=h_{aija}-h_{aiaj}\\
&=\lambda_a\lambda_i(\lambda_j-\lambda_a)
(\delta_{ij} - \delta_{ia}\delta_{ja}).
\end{aligned}
\]
The term with $a=i=j$ vanishes, and $\delta_{ij}$ allows $\lambda_j$ to be replaced by $\lambda_i$. Thus multiplication by $\lambda_a$ and summation yield
\[
\begin{aligned}
\sum_a\lambda_a(h_{ijaa}-h_{aaij})
&=\delta_{ij}\lambda_i\sum_a\lambda_a^2(\lambda_i-\lambda_a)\\
&=(S\lambda_i^2-f_3\lambda_i)\delta_{ij}
=(Sh^2-f_3h)_{ij}.
\end{aligned}
\]
Consequently
\begin{equation}\label{eq:Dhh}
D_hh=-G+Sh^2-f_3h.
\end{equation}
When the derivative of \eqref{eq:Dhh} is taken, the coefficient of $D_h$ must also be differentiated:
\begin{equation}\label{eq:commuteW}
(D_h\nabla h)_{ijk}
=\nabla_k(D_hh)_{ij}-\sum_{a,b}h_{abk}h_{ijab}+W_{ijk},
\quad W_{ijk}=\sum_a\lambda_a(h_{ij,kaa}-h_{ij,aak}).
\end{equation}
The weighted commutator $W$ is obtained by multiplying \eqref{eq:third-order-ricci} by $\lambda_a$ and summing over $a$. The curvature derivatives in that formula are computed from Gauss; in particular,
\[
\nabla_aR^p{}_{ika}
=h_{pka}h_{ia}+h_{pk}h_{iaa}-h_{paa}h_{ik}-h_{pa}h_{ika}.
\]
The two curvature-derivative terms contract to $2(B_3-A_3)$, the two terms with coefficient $2R$ to $4B_3$, and the final curvature term to $SC-A_3$. The first and last contributions can be seen from
\[
\begin{aligned}
\sum_{ijkap}\lambda_a h_{ijk}(\nabla_aR^p{}_{ika})h_{pj}
&=\sum_{ijk}(\lambda_i^2\lambda_j-\lambda_j^3)h_{ijk}^2=B_3-A_3,\\
\sum_{ijkap}\lambda_a h_{ijk}R^p{}_{aka}h_{ijp}
&=\sum_{ijk}(S\lambda_k-\lambda_k^3)h_{ijk}^2=SC-A_3.
\end{aligned}
\]
For the middle type, the relation $R^p{}_{ika}=h_{pk}h_{ia}-h_{pa}h_{ik}$ and $\sum_a\lambda_a h_{aaj}=0$ give directly
\[
2\sum_{i,j,k,a,p}\lambda_a h_{ijk}R^p{}_{ika}h_{pja}
=2\sum_{i,j,k}\lambda_i^2\lambda_k h_{ijk}^2=2B_3;
\]
the partner with $i,j$ interchanged is equal by Codazzi symmetry. Adding the five contributions gives
\begin{equation}\label{eq:Wcontract}
\sum_{ijk}h_{ijk}W_{ijk}=SC+6B_3-3A_3.
\end{equation}

The two terms containing second derivatives of $h$ in the differentiated identity satisfy
\[
\begin{aligned}
\sum_{ijk}h_{ijk}\nabla_kG_{ij}
&=\sum_{ijkab}h_{ijk}(h_{abik}h_{abj}+h_{abi}h_{abjk})=2\mathscr J,\\
\sum_{ijkab}h_{ijk}h_{abk}h_{ijab}&=\mathscr J,
\end{aligned}
\]
by symmetry of $h_{ijk}$ and $h_{ijkl}$ in their first three indices, followed by relabeling. They contribute $-3\mathscr J$. The derivative of $Sh^2$ contributes $2SC$, since $S$ is constant and $(\nabla_kh^2)_{ij}=(\lambda_i+\lambda_j)h_{ijk}$ in a principal frame. Finally, $\sum_{ij}h_{ij}h_{ijk}=0$ eliminates the $\nabla f_3$ term, leaving $-f_3N$ from the derivative of $-f_3h$. The product rule and \eqref{eq:Wcontract} therefore give
\[
\frac12D_hN=\sum_{ijkl}\lambda_lh_{ijkl}^2-3\mathscr J+3SC-f_3N+6B_3-3A_3.
\]
By \eqref{eq:constantS}, $N=S(S-1)$ is constant and the left side vanishes. Rearrangement proves \eqref{eq:fourth}.
\end{proof}

\begin{proposition}\label{prop:boundedposition}
If $S>1$ is constant on a complete three-dimensional self-shrinker, then $\sup_M|X|<\infty$.
\end{proposition}
\begin{proof}
By \eqref{eq:constantS}, $N=S(S-1)>0$. In a principal frame,
\[
A-2B=\sum_{ijk}(\lambda_i^2-2\lambda_i\lambda_j)h_{ijk}^2,
\]
and $|\lambda_i|\le\sqrt S$ gives $|A-2B|\le3SN$. The second derivative identity \eqref{eq:pengterng} therefore implies $|\nabla^2h|^2\le N(10S-2)$. Since $\Delta h-\nabla_{X^T}h=(1-S)h$ and $|\Delta h|\le\sqrt3|\nabla^2h|$, we also have
$|\nabla_{X^T}h|\le\sqrt3|\nabla^2h|+(S-1)\sqrt S$.
We claim that there exist $R,\eta>0$ such that
\begin{equation}\label{eq:gradient-away}
|X^T|\ge R\quad\Longrightarrow\quad|\nabla H|\ge\eta.
\end{equation}
If not, choose $p_j$ with $r_j=|X^T(p_j)|\to\infty$ and $|\nabla H(p_j)|\to0$. At $p_j$ set $e_3=X^T/r_j$ and diagonalize $h|_{e_3^\perp}$. By \eqref{eq:gradH}, $h(e_3,\cdot)=r_j^{-1}\nabla H\to0$, while $\nabla_{e_3}h=r_j^{-1}\nabla_{X^T}h\to0$. After passing to a subsequence, both $h$ and $\nabla h$ converge, with $h\to\operatorname{diag}(\lambda,\mu,0)$. By the symmetry of $\nabla h$, the limiting cubic has no component with index $3$; its squared norm remains $N>0$. The identities $\nabla H\to0$ and $\nabla S=0$ give, for $k=1,2$,
\[
h_{11k}+h_{22k}=0,\qquad\lambda h_{11k}+\mu h_{22k}=0.
\]
If $\lambda\ne\mu$, the $2\times2$ coefficient matrix has determinant $\mu-\lambda\ne0$. Thus $h_{11k}=h_{22k}=0$ for $k=1,2$; by full cubic symmetry these are all the surviving components, contradicting $N>0$. If $\lambda=\mu$, then $\lambda^2=S/2$, and every surviving cubic component lies in the common $\lambda$-eigenspace. Hence $A=B=\lambda^2N=SN/2$, so \eqref{eq:pengterng} becomes
\[
\lim_j|\nabla^2h|^2=N(S-2)+3\left(\frac{SN}{2}-SN\right)
=-N\left(\frac S2+2\right)<0,
\]
another contradiction. Thus \eqref{eq:gradient-away} holds.

Put $H_0=\sqrt{3S}$ and choose $L>2H_0/\eta$. If $|X|$ were unbounded, choose $p$ with $|X(p)|>R+H_0+L$. Follow the unit-speed field $-\nabla H/|\nabla H|$ from $p$. If $\gamma$ is its trajectory, then $|X(\gamma(t))-X(p)|\le t$ because $X$ is isometric on tangent vectors. Since $X=X^T-H\nu$, we have $|X^T|\ge|X|-|H|$. Hence, for $t\le L$ along the existence interval,
\[
|X^T(\gamma(t))|\ge |X(\gamma(t))|-|H(\gamma(t))|
\ge |X(p)|-t-H_0>R.
\]
The field stays smooth there by \eqref{eq:gradient-away}. If the trajectory ceased to exist at some finite time $T<L$, its unit speed would make $\gamma(t)$ Cauchy as $t\uparrow T$; completeness gives a limit point where the same strict inequality permits continuation. Therefore it exists to time $L$. Yet along it $dH/dt=-|\nabla H|\le-\eta$, so $H(\gamma(L))\le H(p)-\eta L<-H_0$, contrary to $|H|\le H_0$.
\end{proof}

\begin{lemma}\label{lem:limit}
If $S>1$ is constant on a complete three-dimensional self-shrinker, there is a common limiting triple $(h,\nabla h,\nabla^2h)$ satisfying \eqref{eq:constantS}, \eqref{eq:pengterng}, and \eqref{eq:fourth}, with
\begin{equation}\label{eq:limit-properties}
\begin{gathered}
h>0,\qquad Hh\ge I,\qquad \sigma_2(h)=\frac{H^2-S}{2}\ge1,\\
\sum_i h_{iik}=\sum_i\lambda_i h_{iik}=0,
\qquad |\nabla h|^2=N=S(S-1)>0.
\end{gathered}
\end{equation}
\end{lemma}
\begin{proof}
Gauss gives a uniform lower sectional-curvature bound $-S/2$. By Proposition \ref{prop:boundedposition}, $\rho=|X|^2/2$ is bounded above. Apply Lemma \ref{lem:omori} to $-\rho$. In view of \eqref{eq:rho}, this supplies points $p_j$ for which $|X^T(p_j)|\to0$ and
\[
H(p_j)h(p_j)\ge(1-\varepsilon_j)I.
\]
The arrays $h,\nabla h,\nabla^2h$ are uniformly bounded by \eqref{eq:constantS} and \eqref{eq:pengterng}; choosing principal frames and taking one subsequence gives a common algebraic limit. In that limit $Hh\ge I$; taking its trace gives $H^2\ge3$. After reversing the normal if necessary, $H>0$, and each eigenvalue satisfies $H\lambda_i\ge1$, so $h>0$. Let $0<a\le b\le c$ be its limiting principal curvatures. The smallest eigenvalue of $Hh\ge I$ gives $Ha\ge1$, while
\[
\sigma_2(h)=ab+ac+bc,\qquad
\sigma_2(h)-Ha=bc-a^2\ge0.
\]
Consequently $\sigma_2(h)\ge1$.

Finally, the ordinary trace of the cubic is $\sum_i h_{iik}=H_k$. Since $\nabla H=h(X^T,\cdot)$ and $X^T(p_j)\to0$, it vanishes in the limit. The weighted trace $\sum_i\lambda_i h_{iik}=S_k/2$ already vanishes at every point. The identities \eqref{eq:constantS}, \eqref{eq:pengterng} and \eqref{eq:fourth} involve only the bounded arrays $h,\nabla h,\nabla^2h$ and consequently pass to this same subsequence.
\end{proof}

\section{The fourth-order comparison}

At the common limit of Lemma \ref{lem:limit}, the shape operator is positive definite with $\sigma_2\ge1$, and the symmetric cubic tensor $h_{ijk}$ satisfies
\[
\sum_i h_{iik}=0,\qquad\sum_i\lambda_i h_{iik}=0.
\]
Ricci commutation determines the nonsymmetric part of $\nabla^2h$. To compare the ordinary and weighted fourth-order identities, we use its complete symmetrization:
\begin{equation}\label{eq:U-definition}
U_{ijkl}=\frac14(h_{ijkl}+h_{jkli}+h_{klij}+h_{lijk}),
\qquad \calC=\nabla^2h-U.
\end{equation}
For symmetric two-tensors $A_0,B_0$, define the six-term symmetric product by
\begin{equation}\label{eq:sym-product}
\begin{aligned}
6(A_0\odot B_0)_{ijkl}={}&(A_0)_{ij}(B_0)_{kl}+(A_0)_{ik}(B_0)_{jl}+(A_0)_{il}(B_0)_{jk}\\
&+(A_0)_{jk}(B_0)_{il}+(A_0)_{jl}(B_0)_{ik}+(A_0)_{kl}(B_0)_{ij}.
\end{aligned}
\end{equation}
For a covariant fourth-order tensor $V$, define the ordinary and $h$-weighted contractions, both of which are two-tensors, and the weighted quadratic form by
\begin{equation}\label{eq:traces-and-weight}
(\tr_I V)_{kl}=\sum_iV_{iikl},\qquad
(L_hV)_{kl}=\sum_{i,j}h_{ij}V_{ijkl},\qquad
|V|_h^2=\sum_{ijkl}\lambda_lV_{ijkl}^2.
\end{equation}
In $L_h(\nabla^2h)$, the contraction is over the first two tensor slots, whereas in $D_hh$ it is over the last two derivative slots. When $h>0$, the corresponding inner product is $\langle V,W\rangle_h=\sum_{i,j,k,l}\lambda_lV_{ijkl}W_{ijkl}$.

Define
\begin{equation}\label{eq:P-R}
\begin{aligned}
P_{ijkl}&=\frac13\sum_m(h_{ijm}h_{klm}+h_{ikm}h_{jlm}+h_{ilm}h_{jkm}),\\
\calR&=\frac32(h^2\odot h^2-h^3\odot h).
\end{aligned}
\end{equation}
The tensor $P$ is the complete symmetrization of the quadratic factor in $\mathscr J$. The tensor $\calR$ accounts for the weighted interaction between $U$ and the Ricci commutation correction $\calC$.

\begin{lemma}\label{lem:algebra}
Under the two trace conditions above,
\begin{equation}\label{eq:three-identities}
|P|^2=\frac{N^2-|G|^2}{3},\qquad
\tr_I P=\frac23G,\quad L_hP=\frac23F,\qquad
\langle P,\calR\rangle=\Lambda_h(G),
\end{equation}
where
\begin{equation}\label{eq:F-Lambda}
\begin{aligned}
F={}&HG-hG-Gh+\left(C-\frac{NH}{2}\right)I+\frac N2h,\\
\Lambda_h(G)={}&\frac{7S-5H^2}{8}A
+\frac{4H^3-7HS+f_3}{8}C\\
&+\frac{-4H^4+12H^2S-3S^2-5Hf_3}{24}N.
\end{aligned}
\end{equation}
Furthermore, with $d=\det h$,
\begin{equation}\label{eq:CH-contract}
B=-2A+2HC-\sigma_2N,\quad
A_3=HA-\sigma_2C+dN,\quad
B_3=\frac{HB-dN}{2}.
\end{equation}
All these identities remain valid at repeated eigenvalues.
\end{lemma}
\begin{proof}
To compute $|P|^2$, compare the three index pairings in its definition. Their difference
$K_{ijkl}=\sum_m(h_{ikm}h_{jlm}-h_{ilm}h_{jkm})$ has curvature symmetries and Ricci contraction $-G$, since $\sum_i h_{iim}=0$. In dimension three the Ricci tensor determines the full curvature tensor, and hence
\begin{equation}\label{eq:algebra-curvature}
K_{ijkl}=-G_{ik}\delta_{jl}-G_{jl}\delta_{ik}
+G_{il}\delta_{jk}+G_{jk}\delta_{il}
+\frac N2(\delta_{ik}\delta_{jl}-\delta_{il}\delta_{jk}).
\end{equation}
The three-dimensional curvature decomposition gives $|K|^2=4|G|^2-N^2$. Each of the three pairings in $P$ has squared norm $|G|^2$; for example,
\[
\sum_{i,j,k,l}\left(\sum_m h_{ijm}h_{klm}\right)^2
=\sum_{m,p}\left(\sum_{i,j}h_{ijm}h_{ijp}\right)^2=|G|^2.
\]
The mixed inner product of any two distinct pairings has the same value $J_0$, by index symmetry. Since $K$ is the difference between two pairings,
\[
9|P|^2=3|G|^2+6J_0,\qquad |K|^2=2|G|^2-2J_0.
\]
Eliminating $J_0$ gives $3|P|^2=3|G|^2-|K|^2=N^2-|G|^2$.
For the ordinary trace, two of the three terms defining $P$ give $G_{kl}$ and the remaining term contains $\sum_ih_{iim}=0$; hence $\tr_IP=2G/3$. Similarly,
\[
(L_hP)_{kl}=\frac23\sum_{i,m}\lambda_i h_{ikm}h_{ilm},
\]
where the remaining term vanishes by $\sum_i\lambda_i h_{iim}=0$. Contracting \eqref{eq:algebra-curvature} with $\lambda_i$ in the $(i,k,i,l)$ positions gives
\[
\sum_{i,m}\lambda_i h_{ikm}h_{ilm}
=(H-\lambda_k-\lambda_l)G_{kl}
+\left(C-\frac{NH}{2}\right)\delta_{kl}+\frac N2h_{kl}=F_{kl}.
\]

To compute the remaining contractions, we linearize the three-dimensional Cayley--Hamilton identity $h^3-Hh^2+\sigma_2h-dI=0$ in the direction $h_k=\nabla_{e_k}h$, with $k$ fixed. The trace conditions give $\tr h_k=\tr(hh_k)=0$, so the variations of $H$ and $\sigma_2$ vanish, while
\[
\delta(\det h)=\tr\big[(h^2-Hh+\sigma_2I)h_k\big]
=\langle h^2,h_k\rangle.
\]
The linearized identity is
\begin{equation}\label{eq:linearCH}
h^2h_k+hh_kh+h_kh^2-H(hh_k+h_kh)+\sigma_2h_k-\langle h^2,h_k\rangle I=0.
\end{equation}
At a principal frame, the $(i,j)$ component is
\[
\big[\lambda_i^2+\lambda_i\lambda_j+\lambda_j^2
-H(\lambda_i+\lambda_j)+\sigma_2\big]h_{ijk}
-\delta_{ij}\langle h^2,h_k\rangle=0.
\]
Multiplication by $h_{ijk}$ and summation gives $2A+B-2HC+\sigma_2N=0$: the last term vanishes because $\sum_i h_{iik}=0$. Multiplication by $\lambda_i h_{ijk}$ instead gives $A_3+2B_3-H(A+B)+\sigma_2C=0$. Using $h^3=Hh^2-\sigma_2h+dI$ proves \eqref{eq:CH-contract}. The contraction $\langle P,\calR\rangle$ also involves $D_*=\sum_k\langle h^2,h_k\rangle^2$. We first express this quantity in terms of $A,C,N$:
\begin{equation}\label{eq:Dstar}
4D_*=-(3S-H^2)A+(3f_3-HS)C+(S^2-Hf_3)N.
\end{equation}
For nonscalar $h$, the two trace conditions say that the diagonal of each $h_k$ is orthogonal to both $(1,1,1)$ and $(\lambda_1,\lambda_2,\lambda_3)$. It is therefore a multiple of $w=(\lambda_2-\lambda_3,\lambda_3-\lambda_1,\lambda_1-\lambda_2)$. Write $h_{iik}=w_i\alpha_k$ and $p=w_1w_2w_3$. The only other independent cubic component is $h_{123}$, and symmetry gives
\[
G_{ii}=|w|^2\alpha_i^2+2w_i^2\sum_{k\ne i}\alpha_k^2+2h_{123}^2.
\]
The relations $\sum_iw_i=0$, $\sum_iw_i^3=3p$ and $|w|^2w_i-2w_i^3=-2p$ give, upon multiplication of the preceding expression for $G_{ii}$ by $w_i$ and summation,
$\sum_iw_iG_{ii}=4p\sum_k\alpha_k^2$. Also $\sum_i\lambda_i^2w_i=-p$, so $\langle h^2,h_k\rangle=-p\alpha_k$. Hence
\[
D_*=p^2\sum_k\alpha_k^2=\frac p4\sum_iw_iG_{ii}.
\]
Expanding the quadratic polynomial in $\lambda_i$ yields
\[
pw_i=-(3S-H^2)\lambda_i^2+(3f_3-HS)\lambda_i+S^2-Hf_3,
\]
and \eqref{eq:Dstar} follows on using $\sum_i\lambda_i^2G_{ii}=A$, $\sum_i\lambda_iG_{ii}=C$ and $\tr G=N$. If $h$ is scalar, both sides vanish directly.

For the remaining contraction $\langle P,\calR\rangle$, only the pair components contribute. In a principal frame they are
\[
3P_{iijj}=\sum_m(h_{iim}h_{jjm}+2h_{ijm}^2),
\qquad \calR_{iijj}=-\frac14\lambda_i\lambda_j(\lambda_i-\lambda_j)^2.
\]
Since each pair $(i,i,j,j)$ has six equal components in the fully symmetric tensors, their contraction begins with
\[
\langle P,\calR\rangle
=-\frac32\sum_{i<j}\lambda_i\lambda_j(\lambda_i-\lambda_j)^2P_{iijj}.
\]
To evaluate this contraction, fix $m$ and put $a_i=\lambda_i h_{iim}$. The weighted trace condition gives $\sum_i a_i=0$, and hence
\[
\sum_{i<j}(\lambda_i-\lambda_j)^2a_ia_j
=-\left(\sum_i\lambda_i a_i\right)^2.
\]
After summing over $m$, these terms in $3P_{iijj}$ contribute $D_*/2$. For the remaining terms, the symmetry $h_{ijm}=h_{jim}$ gives
\[
-\sum_{i<j,m}\lambda_i\lambda_j(\lambda_i-\lambda_j)^2h_{ijm}^2
=\sum_{i,j,m}(\lambda_i^2\lambda_j^2-\lambda_i^3\lambda_j)h_{ijm}^2.
\]
Consequently,
\[
\langle P,\calR\rangle=\frac12D_*+\sum_{ijk}(\lambda_i^2\lambda_j^2-\lambda_i^3\lambda_j)h_{ijk}^2.
\]
A second contraction of \eqref{eq:linearCH}, now with $\lambda_i^2h_{ijk}$, gives
\[
\sum_{ijk}\lambda_i^2\lambda_j^2h_{ijk}^2=D_*+\sigma_2B-2dC,
\quad\sum_{ijk}\lambda_i^3\lambda_jh_{ijk}^2=HB_3-\sigma_2B+dC.
\]
It follows that
\begin{equation}\label{eq:P-R-intermediate}
\langle P,\calR\rangle=\frac32D_*+2\sigma_2B-3dC-HB_3.
\end{equation}
Substitute \eqref{eq:CH-contract}, \eqref{eq:Dstar}, $\sigma_2=(H^2-S)/2$, and $3d=f_3+H^3/2-3HS/2$. Collecting the coefficients of $A,C,N$ gives exactly \eqref{eq:F-Lambda}.
\end{proof}

For the commutation correction $\calC=\nabla^2h-U$, Ricci commutation in a principal frame gives
\[
h_{abk\ell}-h_{ab\ell k}
=\lambda_a\lambda_b(\lambda_k-\lambda_\ell)
(\delta_{a\ell}\delta_{bk}+\delta_{ak}\delta_{b\ell}).
\]
Since the first three indices of $h_{ijkl}$ are symmetric, its six entries with index multiset $(i,i,j,j)$ have only two possible values, $h_{iijj}$ and $h_{jjii}$. Their average is $U_{iijj}$, so
\[
\calC_{iijj}=h_{iijj}-U_{iijj}
=\frac12(h_{iijj}-h_{jjii})
=\frac12\lambda_i\lambda_j(\lambda_i-\lambda_j).
\]
All other index patterns are fully symmetric by Ricci commutation. For $i<j$, write $a_{ij}=\frac12\lambda_i\lambda_j(\lambda_i-\lambda_j)$. The six entries in the pair block $(i,i,j,j)$ are
\[
\calC_{iijj}=\calC_{ijij}=\calC_{jiij}=a_{ij},\quad
\calC_{jjii}=\calC_{jiji}=\calC_{ijji}=-a_{ij}.
\]
For the fully symmetric $\calR$, the only potentially nonzero components are also pair blocks, and all six permutations have the common value
\[
\calR_{iijj}=-\frac14\lambda_i\lambda_j(\lambda_i-\lambda_j)^2.
\]
The entries with four equal indices vanish, and summation over the six permutations in each pair gives
\begin{equation}\label{eq:correction-norms}
\begin{gathered}
c_0:=|\calC|^2=\frac32(Sf_4-f_3^2),\qquad
c_1:=\sum_{ijkl}\lambda_l\calC_{ijkl}^2=\frac34(Sf_5-f_3f_4),\\
r_0:=|\calR|^2=\frac38(Sf_6-4f_3f_5+3f_4^2),\\
D_0:=\frac12(Sh^2-f_3h)=-L_h\calC,\qquad
Y_0:=L_h\calR=\frac14(2f_3h^2-f_4h-Sh^3).
\end{gathered}
\end{equation}
Counting the three permutations ending in $i$ and the three ending in $j$ gives
\[
\begin{aligned}
c_1&=3\sum_{i<j}(\lambda_i+\lambda_j)a_{ij}^2\\
&=\frac34\sum_{i<j}\lambda_i^2\lambda_j^2
(\lambda_i-\lambda_j)(\lambda_i^2-\lambda_j^2)
=\frac34(Sf_5-f_3f_4).
\end{aligned}
\]
For example, $c_0=6\sum_{i<j}a_{ij}^2$ and
$Sf_4-f_3^2=\sum_{i<j}\lambda_i^2\lambda_j^2(\lambda_i-\lambda_j)^2$; the $r_0$ formula follows similarly from the fourth powers of the differences. The opposite signs within each pair block give $\langle U,\calC\rangle=0$. With the last-index weights retained, their contributions instead differ by $\lambda_j-\lambda_i$, so
\begin{equation}\label{eq:weighted-cross}
\langle U,\calC\rangle=0,\qquad
\sum_{ijkl}\lambda_lU_{ijkl}\calC_{ijkl}
=3\sum_{i<j}(\lambda_j-\lambda_i)a_{ij}U_{iijj}
=\langle U,\calR\rangle.
\end{equation}

Differentiating $S$ twice gives $L_h(\nabla^2h)=-G$, hence $L_hU=-G+D_0$. Moreover, the first three indices of $h_{ijkl}$ are symmetric, so exchanging dummy indices shows that the three pairings in $P$ have identical contractions with $\nabla^2h$. Thus
\[
\begin{aligned}
\mathscr J&=\frac13\sum_{ijklm}h_{ijkl}
\big(h_{ijm}h_{klm}+h_{ikm}h_{jlm}+h_{ilm}h_{jkm}\big)\\
&=\langle\nabla^2h,P\rangle=\langle U,P\rangle.
\end{aligned}
\]
The last equality follows from the complete symmetry of $P$. Put
\begin{equation}\label{eq:aux-fourth}
\begin{aligned}
\Xi&=-G+D_0,\qquad Z=\frac32P-\calR,\\
M_b&=N(S-2)+3(A-2B),\\
R_b&=9HA-\frac32(S+5H^2)C
+\frac12(6f_3+5H^3-9HS)N.
\end{aligned}
\end{equation}
The identities \eqref{eq:fourth} and \eqref{eq:CH-contract} give $R_b$. The mixed terms in the weighted identity combine as
\[
-3\mathscr J+2\langle U,\calR\rangle
=-2\langle U,\tfrac32P-\calR\rangle=-2\langle U,Z\rangle.
\]
Equations \eqref{eq:pengterng} and \eqref{eq:fourth} therefore give
\begin{equation}\label{eq:U-constraints}
L_hU=\Xi,\qquad |U|^2+c_0=M_b,\qquad
|U|_h^2-2\langle U,Z\rangle+c_1-R_b=0.
\end{equation}

The ordinary trace of $U$ supplies the second tensor constraint. By the symmetry of $h_{ijkl}$ in its first three indices,
\[
4\sum_iU_{iikl}
=\sum_i(h_{iikl}+h_{ikli}+h_{klii}+h_{liik})
=2H_{,kl}+2(\Delta h)_{kl}.
\]
Thus
\begin{equation}\label{eq:trace-U-prelim}
\tr_IU=\frac12(\nabla^2H+\Delta h).
\end{equation}
Along the common limiting sequence, $X^T\to0$ and $|\nabla h|$ is bounded, so \eqref{eq:gradH} gives $\nabla^2H=h-Hh^2$. The Simons equation then yields $\Delta h=(1-S)h$, and hence
\begin{equation}\label{eq:ordinary-traces}
\Theta:=\tr_IU=\left(1-\frac S2\right)h-\frac H2h^2,
\qquad \tr_I\calR=\frac14(2Sh^2-f_3h-Hh^3).
\end{equation}
For the second equality, the pair components give
$\sum_i\calR_{iikk}=-\frac14\lambda_k\sum_i\lambda_i(\lambda_i-\lambda_k)^2
=\frac14(2S\lambda_k^2-f_3\lambda_k-H\lambda_k^3)$.
For a completely symmetric fourth-order tensor $V$, we use the minimal-norm trace decomposition of Guo and the first author \cite[Corollary~1.4]{GuoGuan2021}. Set $D=\tr_IV$ and $\tau=\tr D$. By \eqref{eq:sym-product},
\begin{equation}\label{eq:tr-product}
\tr_I(I\odot D)=\frac76D+\frac\tau6I,\qquad
\tr_I(I\odot I)=\frac53I.
\end{equation}
Consequently,
\begin{equation}\label{eq:trace-projection}
V^\circ=V-\frac67I\odot D+\frac{3\tau}{35}I\odot I,
\qquad \tr_IV^\circ=0.
\end{equation}
The trace-free part is orthogonal to $I\odot D'$ for every symmetric bilinear form $D'$. Pairing \eqref{eq:trace-projection} with $V$ therefore gives
\begin{equation}\label{eq:trace-norm}
|V|^2=|V^\circ|^2+\frac67|\tr_IV|^2
-\frac3{35}(\tr\tr_IV)^2.
\end{equation}
For a symmetric bilinear form $a$, the same superscript means $a^\circ=a-\frac13(\tr a)I$. In three dimensions $\tr((a^\circ)^4)=|a^\circ|^4/2$ by Cayley--Hamilton. Taking $V=a^\circ\odot a^\circ$ in \eqref{eq:trace-norm}, with $|V|^2=2|a^\circ|^4/3$ and $\tr_IV=2(a^\circ)^2/3$, gives
\begin{equation}\label{eq:trace-square}
\big|(a^\circ\odot a^\circ)^\circ\big|^2=\frac{18}{35}|a^\circ|^4.
\end{equation}

Adding $3H/4$ times the ordinary norm identity to the weighted identity in \eqref{eq:U-constraints} produces $(3H/4)|U|^2-2\langle U,Z\rangle$. Completing this square gives the shift $V=U-4Z/(3H)$. Its ordinary and $h$-weighted traces are
\begin{equation}\label{eq:traces-shift}
\begin{aligned}
\Psi&=\Theta-\frac4{3H}(G-\tr_I\calR),\\
Y&=\Xi-\frac4{3H}(F-Y_0),\\
\hatY&=Y-\frac17\big(H\Psi+\tr(h\Psi)I+2h\Psi+2\Psi h\big)
+\frac{\tr\Psi}{35}(HI+2h).
\end{aligned}
\end{equation}
Indeed, Lemma \ref{lem:algebra} gives $\tr_IV=\Psi$ and $L_hV=Y$. Tracing \eqref{eq:trace-projection} against $h$, using
\[
L_h(I\odot\Psi)=\frac16\big(H\Psi+\tr(h\Psi)I+2h\Psi+2\Psi h\big),
\qquad L_h(I\odot I)=\frac13(HI+2h),
\]
gives $L_hV^\circ=\hatY$.

\begin{proposition}\label{prop:obstruction}
For the limiting matrix $G$, define
\begin{equation}\label{eq:Q}
\begin{aligned}
Q_h(G):={}&\frac{3}{20H}(2N+Hf_3-S^2)^2+c_1-R_b+\frac{3H}{4}(c_0-M_b)\\
&+\frac{3H}{4}\left[\frac67|\Psi|^2-\frac3{35}(\tr\Psi)^2
+6H\langle\hatY,h^\circ\rangle-\frac{162H^2}{35}|h^\circ|^4\right]\\
&-\frac4{3H}\left[\frac34(N^2-|G|^2)-3\Lambda_h(G)+r_0\right].
\end{aligned}
\end{equation}
Put $E=I\odot I$ and $V=U-4Z/(3H)$. The two norm identities in \eqref{eq:U-constraints} give the following exact decomposition:
\begin{equation}\label{eq:two-squares}
\begin{aligned}
Q_h(G)={}&-\left|U-\frac{3\tr\Xi}{5H}E\right|_h^2
-\frac{3H}{4}\left|V^\circ-3H(h^\circ\odot h^\circ)^\circ\right|^2
\le0.
\end{aligned}
\end{equation}
\end{proposition}
\begin{proof}
The metric square $E$ satisfies $E_{iiii}=1$ and $E_{iijj}=1/3$ for $i\ne j$, together with all index permutations. Since $U$ is fully symmetric and $L_hU=\Xi$, direct summation gives
\[
\langle U,E\rangle_h=\sum_{i,k}\lambda_iU_{iikk}=\tr\Xi,\qquad
|E|_h^2=\sum_i\lambda_i+\frac13\sum_{i<j}(\lambda_i+\lambda_j)=\frac{5H}{3}.
\]
The $h$-orthogonal projection of $U$ onto the span of $E$ is therefore $(3\tr\Xi/(5H))E$. Pythagoras gives
\begin{equation}\label{eq:metric-square}
|U|_h^2=\left|U-\frac{3\tr\Xi}{5H}E\right|_h^2
+\frac3{5H}(\tr\Xi)^2.
\end{equation}
As $\Xi=-G+\frac12(Sh^2-f_3h)$ and $\tr G=N$,
\[
\tr\Xi=-N+\frac12(S^2-Hf_3),
\]
so the scalar term in \eqref{eq:metric-square} is the first summand of \eqref{eq:Q}. To combine the two norm identities in \eqref{eq:U-constraints}, use
\[
\theta|U|^2-2\langle U,Z\rangle
=\theta\left|U-\theta^{-1}Z\right|^2-\theta^{-1}|Z|^2,
\qquad \theta=\frac{3H}{4}.
\]
Multiplying $|U|^2+c_0-M_b=0$ by $\theta$ and adding it to the weighted identity gives
\begin{equation}\label{eq:pre-obstruction}
0=|U|_h^2+\frac{3H}{4}\left|U-\frac4{3H}Z\right|^2
-\frac4{3H}|Z|^2+c_1-R_b+\frac{3H}{4}(c_0-M_b).
\end{equation}

Now $\tr_IV=\Psi$ and $L_hV^\circ=\hatY$ by \eqref{eq:traces-shift}. From \eqref{eq:trace-norm}, $|V|^2=|V^\circ|^2+\frac67|\Psi|^2-\frac3{35}(\tr\Psi)^2$. Put $W=(h^\circ\odot h^\circ)^\circ$. The symmetry of $V^\circ$ and $\tr_I V^\circ=0$ give
\[
\begin{aligned}
\langle V^\circ,W\rangle
&=\sum_{i,j,k,l}V^\circ_{ijkl}h^\circ_{ij}h^\circ_{kl}\\
&=\sum_{k,l}\left(\sum_{i,j}h_{ij}V^\circ_{ijkl}\right)h^\circ_{kl}
=\langle L_hV^\circ,h^\circ\rangle=\langle\hatY,h^\circ\rangle.
\end{aligned}
\]
In the second equality, the contribution from $h^\circ_{ij}=h_{ij}-H\delta_{ij}/3$ vanishes by $\tr_I V^\circ=0$.
Combining this with $|(h^\circ\odot h^\circ)^\circ|^2=18|h^\circ|^4/35$ from \eqref{eq:trace-square}, we obtain
\[
\begin{aligned}
|V|^2={}&|V^\circ-3HW|^2
+\frac67|\Psi|^2-\frac3{35}(\tr\Psi)^2\\
&+6H\langle\hatY,h^\circ\rangle-\frac{162H^2}{35}|h^\circ|^4.
\end{aligned}
\]
Finally, by \eqref{eq:three-identities} and \eqref{eq:correction-norms},
\[
|Z|^2=\frac94|P|^2-3\langle P,\calR\rangle+|\calR|^2
=\frac34(N^2-|G|^2)-3\Lambda_h(G)+r_0.
\]
Substituting these expressions and \eqref{eq:metric-square} into \eqref{eq:pre-obstruction}, the remaining terms are exactly $Q_h(G)$ as defined in \eqref{eq:Q}. This proves \eqref{eq:two-squares}, since $h>0$ and $H>0$.
\end{proof}

\begin{lemma}\label{lem:positiveQ}
Let $h>0$ be a three-dimensional symmetric endomorphism satisfying $\sigma_2(h)\ge1$, and put $N=S(S-1)$. For any symmetric matrix $G$ with $\tr G=N$, define $A=\tr(h^2G)$, $C=\tr(hG)$, $B=-2A+2HC-\sigma_2N$, and $Q_h(G)$ by \eqref{eq:Q} and the preceding algebraic formulas. Then
\[
Q_h(G)\ge0.
\]
Equality holds if and only if $h=I/\sqrt3$ and $G=0$.
\end{lemma}
\begin{proof}
Fix $h$ and allow $G$ to range over all symmetric matrices with $\tr G=N$. The quantities $H,S,f_m,N$ and all contractions involving $\calR$ alone are then fixed. By \eqref{eq:traces-shift},
\[
\Psi=\Psi_0-\frac4{3H}G,\qquad
\Psi_0=\Theta+\frac4{3H}\tr_I\calR.
\]
In particular, $\tr\Psi=\tr\Psi_0-4N/(3H)$ is fixed. The formulas for $\Theta$ and $\tr_I\calR$ in \eqref{eq:ordinary-traces} first give
\[
\Psi_0=\left(1-\frac S2-\frac{f_3}{3H}\right)h
+\left(-\frac H2+\frac{2S}{3H}\right)h^2-\frac13h^3.
\]
Replacing $h^3$ by $Hh^2-\sigma_2h+(\det h)I$ and using $\sigma_2=(H^2-S)/2$, we obtain
\begin{equation}\label{eq:Psi0-exp}
\Psi_0=\frac{4S-5H^2}{6H}h^2
+\left(1+\frac{H^2}{6}-\frac{2S}{3}-\frac{f_3}{3H}\right)h
-\frac{\det h}{3}I.
\end{equation}
The two terms in \eqref{eq:Q} that are quadratic in $G$ expand as
\[
\begin{aligned}
\frac{3H}{4}\frac67|\Psi|^2
&=\frac8{7H}|G|^2-\frac{12}{7}\langle\Psi_0,G\rangle
+\frac{9H}{14}|\Psi_0|^2,\\
-\frac4{3H}\left[\frac34(N^2-|G|^2)-3\Lambda_h(G)+r_0\right]
&=\frac1H|G|^2+\frac4H\Lambda_h(G)
-\frac{N^2}{H}-\frac{4r_0}{3H}.
\end{aligned}
\]
Their quadratic contributions add up to $15|G|^2/(7H)$. Moreover, with $A=\tr(h^2G)$ and $C=\tr(hG)$, equation \eqref{eq:Psi0-exp} yields
\begin{equation}\label{eq:Psi0-linear}
\langle\Psi_0,G\rangle
=\frac{4S-5H^2}{6H}A
+\left(1+\frac{H^2}{6}-\frac{2S}{3}-\frac{f_3}{3H}\right)C
-\frac{\det h}{3}N.
\end{equation}
The trace constraint makes $(\tr\Psi)^2$ independent of $G$. To determine the remaining linear term, let $\delta G$ be a symmetric variation with $\tr\delta G=0$. The definitions of $\Xi,F,\Psi,Y$ give
\[
\begin{aligned}
\delta\Xi&=-\delta G,\qquad
\delta F=H\delta G-h\delta G-\delta Gh+\tr(h\delta G)I,\\
\delta\Psi&=-\frac4{3H}\delta G,\qquad
\delta Y=-\frac73\delta G+\frac4{3H}(h\delta G+\delta Gh)
-\frac4{3H}\tr(h\delta G)I.
\end{aligned}
\]
Since $\delta(\tr\Psi)=0$ and $\delta\tr(h\Psi)=-4\tr(h\delta G)/(3H)$, the definition of $\hatY$ in \eqref{eq:traces-shift} yields
\begin{equation}\label{eq:hatY-variation}
\delta\hatY=-\frac{15}{7}\delta G
+\frac{12}{7H}(h\delta G+\delta Gh)
-\frac8{7H}\tr(h\delta G)I.
\end{equation}
As $h^\circ=h-HI/3$ and $\tr\delta G=0$,
\[
\begin{aligned}
\langle\delta G,h^\circ\rangle&=\tr(h\delta G),\\
\langle h\delta G+\delta Gh,h^\circ\rangle
&=2\tr(h^2\delta G)-\frac{2H}{3}\tr(h\delta G),
\end{aligned}
\]
and therefore
\[
\delta\langle\hatY,h^\circ\rangle
=\frac{24}{7H}\tr(h^2\delta G)-\frac{23}{7}\tr(h\delta G).
\]
Since $\hatY$ is affine in $G$, this determines its linear contribution. The Cayley--Hamilton contraction $B=-2A+2HC-\sigma_2N$ also gives
\[
M_b=N(S-2)+3(A-2B)=15A-12HC+N(S-2+6\sigma_2).
\]
Apart from $|G|^2$, the dependence of $Q_h$ on $G$ is through $A=\tr(h^2G)$ and $C=\tr(hG)$. Their coefficients are obtained by combining $-12\langle\Psi_0,G\rangle/7$, $9H^2\langle\hatY,h^\circ\rangle/2$, $-R_b-3HM_b/4$, and $4\Lambda_h(G)/H$. Equations \eqref{eq:Psi0-linear}, \eqref{eq:aux-fourth}, \eqref{eq:F-Lambda}, and the variation above give
\[
\begin{aligned}
b&=\frac{10H^2-8S}{7H}+\frac{108H}{7}-\frac{81H}{4}
+\frac{7S-5H^2}{2H}\\
&=-\frac{3(55H^2-22S)}{28H},
\end{aligned}
\]
while the coefficient of $C$ is
\[
\begin{aligned}
c={}&\frac{-12-2H^2+8S+4f_3/H}{7}-\frac{207H^2}{14}\\
&+\frac{3S+33H^2}{2}+\frac{4H^3-7HS+f_3}{2H}\\
={}&\frac{3(16H^3-4HS-8H+5f_3)}{14H}.
\end{aligned}
\]
Thus, with a scalar $\gamma(h)$ independent of $G$, we have
\begin{equation}\label{eq:G-quadratic}
Q_h(G)=\frac{15}{7H}|G|^2+b\,\tr(h^2G)+c\,\tr(hG)+\gamma(h).
\end{equation}

Since $H>0$, the quadratic part of \eqref{eq:G-quadratic} is positive definite on the trace hyperplane. For a trace-preserving variation $\delta G$, the first variation is
\[
\delta Q_h(G)=\left\langle\frac{30}{7H}G+bh^2+ch,\delta G\right\rangle,
\qquad \tr\delta G=0.
\]
As this variation must vanish for every trace-free $\delta G$, the unique critical point $G_*$ satisfies
\[
\left(\frac{30}{7H}G_*+bh^2+ch\right)^\circ=0,
\qquad \tr G_*=N.
\]
Taking trace-free parts gives
\[
(G_*)^\circ=-\frac{7H}{30}\big[b(h^2)^\circ+ch^\circ\big].
\]
Since $\tr G_*=N$, we have $G_*=(N/3)I+(G_*)^\circ$. Substituting $b,c$, we obtain
\begin{equation}\label{eq:Gstar}
G_* =\frac N3I+\frac{55H^2-22S}{40}(h^2)^\circ
-\frac{16H^3-4HS-8H+5f_3}{20}h^\circ.
\end{equation}
For any $G$ of trace $N$, the matrix $G-G_*$ is trace-free, and the cross term in the expansion of $Q_h(G)-Q_h(G_*)$ vanishes by the critical-point equation. Hence
\begin{equation}\label{eq:completeG}
Q_h(G)=\frac{15}{7H}|G-G_*|^2+Q_h(G_*).
\end{equation}
To evaluate the minimum, take $G_0=(N/3)I$ as a reference point in \eqref{eq:G-quadratic}:
\[
\begin{aligned}
Q_h(G)={}&Q_h(G_0)+\frac{15}{7H}|G-G_0|^2\\
&+\big\langle b(h^2)^\circ+ch^\circ,\,G-G_0\big\rangle.
\end{aligned}
\]
Minimizing this expression, or substituting $G_*$, yields the exact scalar identity
\begin{equation}\label{eq:Q-at-minimum}
Q_h(G_*)=Q_h(G_0)-\frac{7H}{60}
\big|b(h^2)^\circ+ch^\circ\big|^2.
\end{equation}
At $G_0$ the contractions entering \eqref{eq:Q} simplify to
\[
|G_0|^2=\frac{N^2}{3},\qquad
A_0=\frac{NS}{3},\quad C_0=\frac{NH}{3},\quad
B_0=\frac{\sigma_2N}{3},\quad
F(G_0)=\frac N6(HI-h).
\]
The square in \eqref{eq:Q-at-minimum} uses only
\[
\begin{gathered}
|(h^2)^\circ|^2=f_4-\frac{S^2}{3},\qquad
\langle(h^2)^\circ,h^\circ\rangle=f_3-\frac{HS}{3},\qquad
|h^\circ|^2=S-\frac{H^2}{3}.
\end{gathered}
\]
It remains to prove that $Q_h(G_*)\ge0$. As $H>0$, write
\[
M=\frac hH,\qquad
\mu_i=\frac{\lambda_i}{H},\qquad
s=\sum_i\mu_i^2,\quad z=\sum_i\mu_i^3,\quad
r=H^{-2},\quad q=\frac{1-s}{2}.
\]
The normalized eigenvalues satisfy $\sum_i\mu_i=1$ and $\mu_i>0$, so $1/3\le s<1$. Weighted Cauchy--Schwarz gives
\[
s^2=\left(\sum_i\mu_i^{3/2}\mu_i^{1/2}\right)^2
\le\left(\sum_i\mu_i^3\right)\left(\sum_i\mu_i\right)=z.
\]
Moreover $\sigma_2(h)=H^2(1-s)/2=H^2q\ge1$ gives $0<r\le q$. Thus it suffices to work on
\[
\frac13\le s<1,\qquad z\ge s^2,\qquad 0<r\le q.
\]

Set $m_j=\tr M^j$, with $m_0=3$, $m_1=1$, $m_2=s$ and $m_3=z$. The three-dimensional Cayley--Hamilton identity and Newton's formula for the determinant give
\[
M^3=M^2-qM+\frac{1-3s+2z}{6}I,
\qquad m_j=m_{j-1}-qm_{j-2}+\frac{1-3s+2z}{6}m_{j-3}.
\]
In particular,
\begin{equation}\label{eq:normalized-moments}
\begin{aligned}
m_4&=\frac{s^2}{2}-s+\frac{4z}{3}+\frac16,\\
m_5&=\frac{5sz-5s+5z+1}{6},\\
m_6&=\frac{s^3}{4}-\frac{3s^2}{4}+sz-\frac s4+\frac{z^2+z}{3}+\frac1{12}.
\end{aligned}
\end{equation}
Thus $f_j=H^jm_j$, and $N/H^4=s(s-r)$. To evaluate \eqref{eq:Q-at-minimum}, set
\[
\alpha=\frac{55-22s}{40},\qquad
\beta=-\frac{16-4s-8r+5z}{20}.
\]
Since $-7H[b(h^2)^\circ+ch^\circ]/30
=H^4[\alpha(M^2)^\circ+\beta M^\circ]$, \eqref{eq:Q-at-minimum} becomes
\[
\begin{aligned}
\frac{Q_h(G_*)}{H^7}={}&\frac{Q_h(G_0)}{H^7}
-\frac{15}{7}\Big[\alpha^2\left(m_4-\frac{s^2}{3}\right)\\
&\qquad+2\alpha\beta\left(z-\frac{s}{3}\right)
+\beta^2\left(s-\frac13\right)\Big].
\end{aligned}
\]
The expressions for $A_0,B_0,C_0,F(G_0)$, together with $f_j=H^jm_j$, likewise reduce $Q_h(G_0)/H^7$ to $s,z,r$. Substitution into \eqref{eq:Q-at-minimum}, with $\phi(s,z,r)=H^{-7}Q_h(G_*)$, gives
\begin{equation}\label{eq:phi-explicit}
\begin{aligned}
\phi(s,z,r)={}&\frac{45s+224}{80}z^2
+\frac{(368s-496)r-422s^2-15s+251}{80}z\\
&+\frac{4s^2-16s+1}{20}r^2
+\frac{-4s^3-49s^2+63s-8}{20}r\\
&+\frac{-220s^4+1932s^3-2547s^2+462s-95}{640}.
\end{aligned}
\end{equation}
To minimize $\phi$ on the region above, we first fix $s,r$ and vary $z\ge s^2$. Differentiation and separation of $z-s^2$ give
\begin{equation}\label{eq:phi-dz}
\begin{aligned}
80\,\partial_z\phi
={}&2(45s+224)(z-s^2)+(368s-496)r\\
&+90s^3+26s^2-15s+251\\
\ge{}&90s^3-158s^2+417s+3>0.
\end{aligned}
\end{equation}
Here $z-s^2\ge0$; since $368s-496<0$ and $r\le(1-s)/2$, the $r$-term is bounded below by $(368s-496)(1-s)/2$. The last polynomial is positive because $417s-158s^2=s(417-158s)>0$ for $1/3\le s<1$. Thus $\phi(s,z,r)\ge\phi(s,s^2,r)$.

Next fix $s$ and vary $r$ in $0<r\le q$, with $z=s^2$. Writing $x=(3s-1)/2\in[0,1)$, we obtain
\begin{equation}\label{eq:phi-dr}
\begin{aligned}
-270\,\partial_r\phi(s,s^2,r)
={}&40+207x+x^2(510-352x)\\
&+27r(16s-4s^2-1)>0.
\end{aligned}
\end{equation}
Indeed, $510-352x>0$, and $16s-4s^2-1>0$ for $s\ge1/3$ in the given range. Therefore $\phi(s,s^2,r)\ge\phi(s,s^2,q)$. Substituting $r=q=(1-x)/3$ in \eqref{eq:phi-explicit} gives
\begin{equation}\label{eq:phi-end}
\phi(s,s^2,q)=\frac{x}{4320}\left[1160-629x+8x^2(489-430x)+320x^4\right]\ge0.
\end{equation}
Both $1160-629x$ and $489-430x$ are positive on $0\le x<1$. We have proved $\phi\ge0$, and \eqref{eq:completeG} gives $Q_h(G)\ge0$.

If $Q_h(G)=0$, then \eqref{eq:completeG}, the nonnegative lower bound for $\phi$, and $H>0$ force $G=G_*$ and equality in both monotonicity comparisons. The strict derivative signs and \eqref{eq:phi-end} give $z=s^2$, $r=q$, and $x=0$. Thus $s=r=1/3$. As $\sum\mu_i=1$ and $\sum\mu_i^2=1/3$, equality in Cauchy--Schwarz forces $\mu_1=\mu_2=\mu_3=1/3$. Hence $H=\sqrt3$ and $h=I/\sqrt3$. This gives $S=1$, $N=0$, and $G_*=0$. Conversely, these values give equality by substitution.
\end{proof}

\begin{proof}[Proof of Theorem \ref{thm:main}]
The identity \eqref{eq:constantS} excludes $0<S<1$. Suppose $S>1$. At the common limit of Lemma \ref{lem:limit}, one has $h>0$ and $\sigma_2(h)\ge1$. Proposition \ref{prop:obstruction} gives $Q_h(G)\le0$, whereas Lemma \ref{lem:positiveQ} gives $Q_h(G)\ge0$, with equality only for $h=I/\sqrt3$, hence $S=1$. This contradicts $S>1$.

Thus $S\in\{0,1\}$, and \eqref{eq:constantS} gives $\nabla h=0$.
By Lawson's classification of Euclidean hypersurfaces with parallel
second fundamental form \cite{Lawson1969}, the image lies in an affine
hyperplane if $S=0$, or in a round generalized cylinder with spherical
factor $\mathbb S^k(r)$, $1\le k\le3$, if $S=1$.
In the latter case, $S=k/r^2=1$ gives $r=\sqrt{k}$.
The self-shrinker equation forces the hyperplane or the cylinder axis
to pass through the origin.

If we passed to a normal double cover, the original immersion has
the same image. The original immersion is therefore a local isometry
into the corresponding model and, by completeness, a Riemannian
covering onto it.
If $M$ is compact, the model must be $\mathbb S^3(\sqrt3)$.
Since this sphere is simply connected and $M$ is connected,
the covering is one-sheeted.
\end{proof}

\end{document}